\documentclass{amsart}

\usepackage{amssymb,amsfonts,amsmath,amscd}
\usepackage[all,arc]{xy}
\usepackage{enumerate}
\usepackage{mathrsfs}
\usepackage{amsthm}
\usepackage{color}
\usepackage{hyperref}
\usepackage{fullpage}

\newtheorem{thm}{Theorem}[section]
\newtheorem{cor}[thm]{Corollary}
\newtheorem{prop}[thm]{Proposition}
\newtheorem{lem}[thm]{Lemma}

\theoremstyle{definition}

\theoremstyle{remark}
\newtheorem{rem}[thm]{Remark}

\DeclareMathOperator{\stab}{stab}
\DeclareMathOperator{\cochar}{cochar}
\DeclareMathOperator{\Leaf}{Leaf}
\DeclareMathOperator{\Slope}{Slope}
\DeclareMathOperator{\supp}{supp}

\title{Restriction formula for stable basis of the Springer resolution}
\author{Changjian Su}
\date{}
\address{Department of Mathematics\\
  Columbia University\\
  2990 Broadway\\
  New York, NY 10027}
\email{changjian@math.columbia.edu}

\begin{document}
\maketitle

\begin{abstract}

We give restriction formula for stable basis of the Springer resolution, and generalize it to cotangent bundles of homogeneous spaces. By a limiting process, we get the restriction formula of Schubert varieties.

\end{abstract}

\section{Introduction}
In \cite{Maulik2012}, Maulik and Okounkov defined the stable basis for a wide class of varieties. In this paper, we apply their general construction to the special case of the Springer resolution and study the restriction of the stable basis to fixed points. 

To state our main result, let us fix some notations. Let $G$ be a complex semisimple algebraic group, $B$ be a Borel subgroup, and $\mathcal{B}$ be its associated flag variety. Let $P$ be a parabolic subgroup containing $B$, and $\mathcal{P}=G/P$ be the partial flag variety. Let $R^+$ be those roots in $B$, and $A$ be a maximal torus of $G$ contained in $B$. Let $\mathbb{C}^*$ dilate the fiber of $T^*\mathcal{B}$ by a nontrivial character $-\hbar$, and $T=A\times \mathbb{C}^*$. It is well-known that the fixed point set $(T^*\mathcal{B})^A$ is in one-to-one correspondence with the Weyl group $W$ of $G$. Let $wB$ denote the fixed point corresponding to $w$. The stable envelope $\stab_\pm(y)\in H_T^*(T^*\mathcal{B})$ will be defined in Section \ref{section 2}. Here, $y$ means the unit in $H_T^*(yB)$, and $+$ is some chamber in the cocharacter lattice. The main theorems are
\begin{thm}\label{restriction -B}
Let $y=\sigma_1\sigma_2\cdots\sigma_l$ be a reduced expression for $y\in W$. Then
\begin{equation}\label{formula restriction -B}
\stab_-(w)|_y=(-1)^{l(y)}\prod\limits_{\alpha\in R^+\setminus R(y)}(\alpha-\hbar)\sum\limits_{\substack{1\leq i_1<i_2<\dots<i_k\leq l\\w=\sigma_{i_1}\sigma_{i_2}\dots\sigma_{i_k}}}\hbar^{l-k}\prod\limits_{j=1}^k \beta_{i_j},
\end{equation}
where $\sigma_i$ is the simple reflection associated to a simple root $\alpha_i$, $\beta_i=\sigma_1\cdots\sigma_{i-1}\alpha_i$, $R(y)=\{\beta_i|1\leq i\leq l\}$, and $\stab_-(w)|_y$ denotes the restriction of $\stab_-(w)$ to the fixed point $yB$.
\end{thm}

For the positive chamber, we have
\begin{thm}\label{restriction B}
Let $y=\sigma_1\sigma_2\cdots\sigma_l$ be a reduced expression for $y\in W$, and $w\leq y$. Then
\[
\stab_+(y)|_w=\sum\limits_{\substack{1\leq i_1<i_2<\dots<i_k\leq l\\ 
w=\sigma_{i_1}\sigma_{i_2}\dots\sigma_{i_k}}}(-1)^l \prod\limits_{j=1}^k\frac{\sigma_{i_1}\sigma_{i_2}\dots\sigma_{i_j}\alpha_{i_j}-\hbar}{\sigma_{i_1}\sigma_{i_2}\dots\sigma_{i_j}\alpha_{i_j}} \frac{\hbar^{l-k}}{\prod\limits_{j=0}^k\prod\limits_{i_j<r<i_{j+1}}\sigma_{i_1}\sigma_{i_2}\dots\sigma_{i_j}\alpha_r}\prod\limits_{\alpha\in R^+}\alpha.
\]
\end{thm}

The stable basis is an interesting object. In our case, it is the characteristic cycles of Verma modules up to a sign, see \cite{ginsburg1986} and Remark 3.5.3 in \cite{Maulik2012}. In the case of Hilbert schemes of points on $\mathbb{C}^2$, it corresponds to Schur functions if we identify the equivariant cohomology ring of Hilbert schemes with the symmetric functions, while the fixed point basis corresponds to Jack symmetric functions, see e.g. \cite{Maulik2012}, \cite{hiraku1999lectures}, \cite{nakajima2014more}. The transition matrix between these two bases was obtained in \cite{Shenfeld2013}. In \cite{smirnov2014polynomials}, Smirnov defined generalized Jack polynomials in the equivariant cohomology ring of instanton moduli space, and he used stable basis to derive a combinatorial formula for the expansion of generalized Jack polynomials in the basis of Schur polynomials. There is also a K-theoretic version of the stable basis. Negut defined some rational version of the Schur polynomials with the help of this K-theoretic stable basis in Hilbert scheme of points on $\mathbb{C}^2$, and he proved a rational version of the Pieri rule (\cite{negut2014m}).

The paper is organized as follows. In Section 2, we apply results in \cite{Maulik2012} to define the stable basis of $T^*\mathcal{B}$ and $T^*\mathcal{P}$. In Section 3, we prove our main Theorems. We give the details for the proof of Theorem \ref{restriction -B}. Theorem \ref{restriction B} can be proved similarly. In Section 4, we give some applications: firstly, we generalize our main theorems to $T^*\mathcal{P}$ case; secondly, we get Billey's restriction formula \cite{Billey1997} for Schubert varieties from our main result. A use of our method gives a simple proof of the restriction formula for Schubert variety of $G/P$ (see \cite{Tymoczko2009}), which can also be deduced from a limiting process. The restriction formulas for stable basis will also be used in a future paper about equivariant quantum cohomology of cotangent bundles of homogeneous spaces (\cite{Su}).

\subsection*{Acknowledgments}  I wish to express my deepest thanks to my advisor Prof. Andrei Okounkov for teaching me stable basis and his patience and invaluable guidance. The author also thanks Chiu-Chu Liu, Michael McBreen, Davesh Maulik, Andrei Negut, Andrey Smirnov, Zijun Zhou, Zhengyu Zong for many stimulating conversations and emails. A lot of thanks also go to my friend Pak-Hin Lee for editing a previous version of the paper.

\section{Stable basis}\label{section 2}

In this section, we apply the construction in \cite{Maulik2012} to $T^*\mathcal{B}$ and $T^*\mathcal{P}$.

Let us fix more notations. Let $\Delta$ be the set of simple roots, and $I$ be a subset of $\Delta$. Let $P_I=\bigcup_{w\in W_I}BwB$ a parabolic subgroup containing $B$, $W_{P_I}$ the subgroup generated by the simple reflections $\sigma_\alpha$ for $\alpha\in I$, and $R_{P_I}^{\pm}$ the roots in $R^{\pm}$ spanned by $I$. It is well-known that every parabolic subgroup is conjugate to some parabolic subgroup containing the fixed Borel subgroup $B$, which is of the form $P_I$ for some subset $I$ in $\Delta$, and  $P_I$ is not conjugate to $P_J$ if the two subsets $I$ and $J$ are not equal (see \cite{Springer2010}). For any  cohomology classes $\alpha$ and $\beta$, let
\[(\alpha,\beta):=\int_X \alpha\cdot \beta\]
denote the standard inner product on cohomology.

\subsection{Stable basis for $T^*\mathcal{B}$}

\subsubsection{Fixed point set}
The $A$-fixed points of $T^*\mathcal{B}$ is in one-to-one correspondence with the Weyl group $W$. The fixed point corresponds to $w\in W$ is denoted by $wB$. For any cohomology class $\alpha\in H_T^*(T^*\mathcal{B})$, let $\alpha|_w$ denote the restriction of $\alpha$ to the fixed point $wB$.

\subsubsection{Chamber decomposition}
The cocharacters
\[
\sigma:\mathbb{C^*}\rightarrow A
\]
form a lattice. Let 
\[
\mathfrak{a}_{\mathbb{R}}=\cochar(A)\otimes_{\mathbb{Z}}\mathbb{R}.
\]

Define the torus roots to be the $A$-weights occurring in the normal bundle to $(T^*\mathcal{B})^A$. Then the root hyperplanes partition $\mathfrak{a}_{\mathbb{R}}$ into finitely many chambers
\[
\mathfrak{a}_{\mathbb{R}}\setminus\bigcup \alpha_i^\perp=\coprod \mathfrak{C}_i.
\]
It is easy to see in this case that the torus roots are just the roots for $G$. Let $+$ denote the chamber such that all roots in $R^+$ are positive on it, and $-$ the opposite chamber.

\subsubsection{Stable leaves}
Let $\mathfrak{C}$ be a chamber. For any fixed point $yB$, define the stable leaf of $yB$ by
\[
\Leaf_{\mathfrak{C}}(yB)=\left\{x\in T^*\mathcal{B} \left| \lim\limits_{z\rightarrow 0}\sigma(z)\cdot x=yB \right. \right\}
\]
where $\sigma$ is any cocharacter in $\mathfrak{C}$; the limit is independent of the choice of $\sigma\in \mathfrak{C}$. In the $T^*\mathcal{B}$ case, $\Leaf_+(yB)=T_{ByB/B}^*\mathcal{B}$, and $\Leaf_-(yB)=T_{B^-yB/B}^*\mathcal{B}$, where $B^-$ is the opposite Borel subgroup.

Define a partial order on the fixed points as follows:
\[
wB\preceq_{\mathfrak{C}} yB \text{\quad if\quad} \overline{\Leaf_{\mathfrak{C}}(yB)}\cap wB\neq \emptyset.
\]
By the description of $\Leaf_+(yB)$, the order $\preceq_+$ is the same as the Bruhat order $\leq$, and $\preceq_-$ is the opposite order.  Define the slope of a fixed point $yB$ by
\[
\Slope_{\mathfrak{C}}(yB)=\bigcup_{wB\preceq_{\mathfrak{C}} yB} \Leaf_{\mathfrak{C}}(wB).
\] 

\subsubsection{Stable basis}
For each $y\in W$, let $T_y^*\mathcal{B}$ and $T_y(T^*\mathcal{B})$ denote $T_{yB}^*\mathcal{B}$ and $T_{yB}(T^*\mathcal{B})$ respectively, and define $\epsilon_y=e^A(T_y^*\mathcal{B})$. Here, $e^A$ denotes the $A$-equivariant Euler class. Let $N_y$ denote the normal bundle of $T^*\mathcal{B}$ at the fixed point $yB$. The chamber $\mathfrak{C}$ gives a decomposition of the normal bundle
\[
N_y=N_{y,+}\oplus N_{y,-}
\]
into $A$-weights which are positive and negative on $\mathfrak{C}$ respectively. The sign in $\pm e(N_{y,-})$ is determined by the condition
\[
\pm e(N_{y,-})|_{H_A^*(\text{pt})}=\epsilon_y.
\]
The following theorem is Theorem 3.3.4 in \cite{Maulik2012} applied to $T^*\mathcal{B}$.
 
\begin{thm}[\cite{Maulik2012}]\label{stable theorem B}
There exists a unique map of $H_T^*(\text{pt})$-modules
\[
\stab_{\mathfrak{C}}:H_T^*((T^*\mathcal{B})^A)\rightarrow H_T^*(T^*\mathcal{B})
\]
such that for any $y\in W$, $\Gamma=\stab_{\mathfrak{C}}(y)$ satisfies:
\begin{enumerate}
\item $\supp \Gamma\subset \Slope_{\mathfrak{C}}(yB)$,
\item $\Gamma|_y=\pm e(N_{-,y})$, with sign according to $\epsilon_y$,
\item $\Gamma|_w$ is divisible by $\hbar$, for any $wB\prec_{\mathfrak{C}} yB$,
\end{enumerate}
where $y$ in $\stab_{\mathfrak{C}}(y)$ is the unit in $H_T^*(yB)$.
\end{thm}

\begin{rem}
\label{remark 2.2}
\leavevmode
\begin{enumerate}
\item
The map is defined by a Lagrangian correspondence between $(T^*\mathcal{B})^A\times T^*\mathcal{B}$, hence maps middle degree to middle degree.
\item 
From the characterization, the transition matrix from $\{\stab_{\mathfrak{C}}(y)| y\in W\}$ to the fixed point basis is a triangular matrix with nontrivial diagonal terms. Hence, after localization, $\{\stab_{\mathfrak{C}}(y)| y\in W\}$ form a basis for the cohomology, which we call the \textbf{stable basis}.
\item
Theorem 4.4.1 in \cite{Maulik2012} proves that $\{\stab_{\mathfrak{C}}(y)| y\in W\}$ and $\{(-1)^n\stab_{\mathfrak{-C}}(y)| y\in W\}$ are dual bases, i.e.,
\[(\stab_{\mathfrak{C}}(y),(-1)^n\stab_{\mathfrak{-C}}(w))=\delta_{y,w}.\]
Here $n=\dim_{\mathbb{C}}\mathcal{B}$.
\end{enumerate}
\end{rem}

\subsection{Stable basis for $T^*\mathcal{P}$}
A similar construction works for $T^*\mathcal{P}$. In this case, the fixed point set $(T^*\mathcal{P})^A$ corresponds to $W/W_P$ (\cite{Bernstein1973a}). Let $yP$ denote the fixed point in $T^*\mathcal{P}$ corresponding to the coset $yW_P$. Let $T_{\bar{y}}^*\mathcal{P}$ and $T_{\bar{y}}(T^*\mathcal{P})$ denote $T_{yP}^*\mathcal{P}$ and $T_{\bar{y}}(T^*\mathcal{P})$, respectively.  Define $\epsilon_{\bar{y}}=e^A(T_{\bar{y}}^*\mathcal{P})$. For any cohomology class $\alpha\in H_T^*(T^*\mathcal{P})$, let $\alpha|_{\bar{y}}$ denote the restriction of $\alpha$ to the fixed point $yP$. Then the theorem is 
\begin{thm}[\cite{Maulik2012}]\label{stable for P}
There exists a unique map of $H_T^*(\text{pt})$-modules
\[
\stab_{\mathfrak{C}}:H_T^*((T^*\mathcal{P})^A)\rightarrow H_T^*(T^*\mathcal{P})
\]
such that for any $\bar{y}\in W/W_P$, $\Gamma=\stab_{\mathfrak{C}}(\bar{y})$ satisfies:
\begin{enumerate}
\item $\supp \Gamma\subset \Slope_{\mathfrak{C}}(yP)$,
\item $\Gamma|_{\bar{y}}=\pm e(N_{-,\bar{y}})$, with sign according to $\epsilon_{\bar{y}}$,
\item $\Gamma|_{\bar{w}}$ is divisible by $\hbar$, for any $\bar{w}\prec_{\mathfrak{C}} \bar{y}$,
\end{enumerate}
where $\bar{y}$ in $\stab_{\mathfrak{C}}(\bar{y})$ is the unit in $H_T^*(yP)$.
\end{thm}

\begin{rem}
The Bruhat order on $W/W_P$ is defined as follows:
\[yW_P<wW_P \text{\quad if\quad } \overline{ByP/P}\subset \overline{BwP/P}.\]
If the chamber $\mathfrak{C}=+$, then the order $\preceq_{+}$ is the Bruhat order on $W/W_P$. If the chamber $\mathfrak{C}=-$, the order is the opposite Bruhat order.
\end{rem}
From now on, we use $\stab_{\mathfrak{C}}(y)$ to denote the stable basis of $T^*\mathcal{B}$, and $\stab_{\mathfrak{C}}(\bar{y})$ to denote the stable basis of $T^*\mathcal{P}$.

\section{Restriction formula of stable basis for $T^*\mathcal{B}$}
In this section, we prove Theorem \ref{restriction -B} and Theorem \ref{restriction B}.

\subsection{Proof of Theorem \ref{restriction -B}}

Let $Q$ be the quotient field of $H_T^*(\text{pt})$, and $F(W,Q)$ be the functions from $W$ to $Q$. Restriction to fixed points gives a map
\[H_T^*(T^*\mathcal{B})\rightarrow H_T^*((T^*\mathcal{B})^T)
=\bigoplus_{w\in W}H_T^*(wB)\]
and embeds $H_T^*(T^*\mathcal{B})$ into $F(W,Q)$.

It is well-known that the diagonal $G$-orbits on $\mathcal{B}\times \mathcal{B}$ are indexed by the Weyl group, see Chapter 3 in \cite{Chriss2009}. For each simple root $\alpha\in \Delta$, let $Y_{\alpha}$ be the orbit corresponding to the reflection $\sigma_{\alpha}$. Then 
\[\overline{Y_{\alpha}}=\mathcal{B}\times_{\mathcal{P}_{\alpha}}\mathcal{B}\]
where $\mathcal{P}_{\alpha}=G/P_{\alpha}$ and $P_{\alpha}$ is the minimal parabolic subgroup corresponding to the simple root $\alpha$. Let $T_{\overline{Y_{\alpha}}}^*(\mathcal{B}\times\mathcal{B})$ be the conormal bundle to $\overline{Y_{\alpha}}$. This is a Lagrangian correspondence in $T^*\mathcal{B}\times T^*\mathcal{B}$, and defines a map
\[
D_{\alpha}:H_T^*(T^*\mathcal{B})\rightarrow H_T^*(T^*\mathcal{B}).
\]
Define an operator 
$A_0: F(W,Q)\rightarrow F(W,Q)$ by the formula
\[(A_0\psi)(w)=\frac{\psi(w\sigma_{\alpha})-\psi(w)}{w\alpha}(w\alpha-\hbar).\]
A similar operator is defined in \cite{Bernstein1973a}. Then we have the following important commutative diagram.
\begin{prop}\label{commutative diagram 1}
The diagram
\[\xymatrix{
H_T^*(T^*\mathcal{B}) \ar@{^{(}->}[r]
 \ar[d]_{D_{\alpha}} & F(W,Q) \ar[d]^{A_0} \\
H_T^*(T^*\mathcal{B}) \ar@{^{(}->}[r] & F(W,Q)\\} \]
commutes.
\end{prop}

\begin{proof}
Since $H_T^*(T^*\mathcal{B})$ has a fixed point basis after localization, it suffices to show that the two paths around the diagram agree on elements of the form $\iota_{y*}(1)$, where $\iota_y$ is the inclusion of the fixed point $yB$ in $T^*\mathcal{B}$, and 1 is the unit in $H_T^*(yB)$. Such an element gives to a function $\psi_y\in F(W,Q)$ characterized by
\[\psi_y(y)=e(T_yT^*\mathcal{B})\]
and $\psi_y(w)=0$ for $w\neq y$.

Then 
\[A_0(\psi_y)(y)= -\frac{y\alpha-\hbar}{y\alpha}\psi_y(y),\quad A_0(\psi_y)(y\sigma_{\alpha})=\frac{y\alpha+\hbar}{y\alpha}\psi_y(y)\]
and 
\[A_0(\psi_y)(w)=0, \text{for } w\notin \{y,y\sigma_{\alpha}\}.\]

Going along the other way of the diagram, we have
\[D_{\alpha}(\iota_{y*}(1))=\sum_{w\in W}\frac{\left(T_{\overline{Y_{\alpha}}}^*(\mathcal{B}\times\mathcal{B}),\iota_{y*}1\otimes \iota_{w*}1\right)}{e(T_wT^*\mathcal{B})}\iota_{w*}(1).\]
By the definition of $Y_{\alpha}$, $\left(T_{\overline{Y_{\alpha}}}^*(\mathcal{B}\times\mathcal{B}),\iota_{y*}1\otimes \iota_{w*}1\right)$ is nonzero if and only if $w\in \{y,y\sigma_{\alpha}\}$.
By localization,
\[\left(T_{\overline{Y_{\alpha}}}^*(\mathcal{B}\times\mathcal{B}),\iota_{y*} 1\otimes \iota_{y*}1\right)=-\frac{y\alpha-\hbar}{y\alpha}e(T_yT^*\mathcal{B})\]
and
\[\left(T_{\overline{Y_{\alpha}}}^*(\mathcal{B}\times\mathcal{B}),\iota_{y*}1\otimes \iota_{y\sigma_{\alpha}*}1\right)=\frac{y\alpha-\hbar}{y\alpha}e(T_{y\sigma_{\alpha}}T^*\mathcal{B}).\]
Hence 
\[D_{\alpha}(\iota_{y*}1)=-\frac{y\alpha-\hbar}{y\alpha}\iota_{y*}1+\frac{y\alpha-\hbar}{y\alpha}\iota_{y\sigma_{\alpha}*}1.\]
Therefore,
\[D_{\alpha}(\iota_{y*}1)(y)=-\frac{y\alpha-\hbar}{y\alpha}\psi_y(y)\]  
and
\[D_{\alpha}(\iota_{y*}1)(y\sigma_{\alpha})=\frac{y\alpha-\hbar}{y\alpha}e(T_{y\sigma_{\alpha}}T^*\mathcal{B})\]
Since $\alpha$ is a simple root,
\begin{align*}
e(T_{y\sigma_{\alpha}}T^*\mathcal{B})&=\prod\limits_{\beta\in R^+}(y\sigma_{\alpha}\beta-\hbar)(-y\sigma_{\alpha}\beta)\\
&=\prod\limits_{\beta\in R^+\setminus \{\alpha\}}(y\sigma_{\alpha}\beta-\hbar)(-y\sigma_{\alpha}\beta)\cdot(-y\alpha-\hbar)y\alpha \\
&=\frac{y\alpha+\hbar}{y\alpha-\hbar}e(T_yT^*\mathcal{B}),
\end{align*}
so we get
\[D_{\alpha}(\iota_{y*}(1))(y\sigma_{\alpha})=\frac{y\alpha+\hbar}{y\alpha}\psi_y(y).\]

Since $D_{\alpha}(\iota_{y*}(1))$ and $A_0(\psi_y)$ take the same values on $W$,
 \[D_{\alpha}(\iota_{y*}(y))=A_0(\psi_y)\]
 as desired.
\end{proof}

The image of the stable basis under the operator $D_{\alpha}$ is given by the following lemma.
\begin{lem}\label{image of stab}
\[D_{\alpha}\stab_\pm(y)=-\stab_\pm(y)-\stab_\pm(y\sigma_{\alpha}).\]
\end{lem}
\begin{proof}
We only prove for the $+$ case; the $-$ case is almost the same.

By Remark \ref{remark 2.2}(3), the lemma is equivalent to 
\[(D_{\alpha}\stab_+(y),(-1)^n\stab_-(w))=
\left\{ \begin{array}{cc}
-1& w\in \{y,y\sigma_{\alpha}\}\\
0 & \text{otherwise}.
\end{array}\right. \]
By the properties of stable basis,
\[(D_{\alpha}\stab_+(y),(-1)^n\stab_-(w))\]
is a proper integral. Hence it lies in the nonlocalized coefficient ring $H_T^*(\text{pt})$. A degree count shows it is actually a constant. Therefore we can let $\hbar=0$. Then a simple localization calculation using properties (2) and (3) of Theorem \ref{stable theorem B} yields the desired result.
\end{proof}

Applying Proposition \ref{commutative diagram 1} to the stable basis $\{\stab_-(w)|w\in W\}$, we have
\begin{cor}\label{characterize -}
The stable basis $\{\stab_-(w)|w\in W\}$ are uniquely characterized by the following properties:
\begin{enumerate}
\item 
$\stab_-(w)|_y=0$, unless $y\geq w$.
\item 
$\stab_-(w)|_w=\prod\limits_{\alpha\in R^+,w\alpha\in R^+}(w\alpha-\hbar)\prod\limits_{\alpha\in R^+,w\alpha\in -R^+}w\alpha$.
\item
For any simple root $\alpha$, and $l(y\sigma_{\alpha})=l(y)+1$,
\[\stab_-(w)|_{y\sigma_{\alpha}}=-\frac{\hbar}{y\alpha-\hbar}\stab_-(w)|_y-\frac{y\alpha}{y\alpha-\hbar}\stab_-(w\sigma_{\alpha})|_y.\] 
\end{enumerate}
\end{cor}
\begin{proof}
It is easy to see that $\{\stab_-(w)|w\in W\}$ satisfies these properties: (1) and (2) follow directly from Theorem \ref{stable theorem B}, and (3) follows from Proposition \ref{commutative diagram 1} and Lemma \ref{image of stab}.

To show the uniqueness of the stable basis satisfying these properties is equivalent to show these properties uniquely determine the values $\stab_-(w)|_y$. We argue by ascending induction on the length $l(y)$ of $y$. Note that $\stab_-(w)|_1$ is determined by (1) and (2). Assume that $l(y\sigma_{\alpha})=l(y)+1$ for some simple root $\alpha$. Then $\stab_-(w)|_{y\sigma_{\alpha}}$ is determined by $\stab_-(w)|_{y}$ and $\stab_-(w\sigma_{\alpha})|_{y}$ by (3), which are known by the induction hypothesis. 
\end{proof}

For the positive chamber, we have
\begin{cor}\label{characterize +}
The stable basis $\{\stab_+(y)|y\in W\}$ are uniquely characterized by the following properties:
\leavevmode
\begin{enumerate}
\item 
$\stab_+(y)|_w=0$, unless $w\leq y$.
\item 
$\stab_+(y)|_y=\prod\limits_{\alpha\in R^+,y\alpha<0}(y\alpha-\hbar)\prod\limits_{\alpha\in R^+,y\alpha>0}y\alpha$.
\item
For any simple root $\alpha$, and $l(y\sigma_{\alpha})=l(y)+1$,
\[\stab_+(y\sigma_{\alpha})|_w=-\frac{\hbar}{w\alpha}\stab_+(y)|_w-\frac{w\alpha-\hbar}{w\alpha}\stab_+(y)|_{w\sigma_{\alpha}}.\]
\end{enumerate}
\end{cor}
The proof is almost the same as the proof of Corollary \ref{characterize -}, so we omit it. We now prove Theorem \ref{restriction -B}. We show that the formula given in Theorem \ref{restriction -B} does not depend on the reduced expression of $y$ and it satisfies the properties in Corollary \ref{characterize -}.

We first give a proof of independence of reduced expression. 

Let $\Lambda$ be the root lattice, and let $\mathcal{A}$ be the algebra over $\mathbb{Q}[\Lambda](\hbar)$ generated by $\{u_w|w\in W\}$, with relations
\[u_wu_y=u_{wy}, u_wf=fu_w,\] 
where $f\in \mathbb{Q}[\Lambda](\hbar)$. For a reduced word $y=\sigma_1\cdots\sigma_l$, define
\[R_{\alpha_1,\cdots,\alpha_l}=\prod\limits_{i=1}^l \left(1+\frac{\beta_i}{\hbar}u_{\sigma_i}\right),\]
where $\beta_i=\sigma_1\cdots\sigma_{i-1}\alpha_i$. Expanding it, we have
\begin{equation}\label{equation P}
R_{\alpha_1,\cdots,\alpha_l}=\frac{\prod\limits_{i=1}^l(\hbar-\beta_i)}{\hbar^l\prod\limits_{\alpha\in R^+} (\alpha-\hbar)}\sum\limits_w \stab_-(w)|_{\sigma_1\cdots\sigma_l}u_w
\end{equation}
where $\stab_-(w)|_{\sigma_1\cdots\sigma_l}$ is given by Theorem \ref{restriction -B}. Thus we only need to prove 
\begin{prop}\label{independence for -}
$R_{\alpha_1,\cdots,\alpha_l}$ does not depend on the reduced expression of $y=\sigma_1\cdots\sigma_l$.
\end{prop}

\begin{rem}
In \cite{Billey1997}, Billey proves this case by case when the Weyl group is replaced by the nil-Coxeter group, which is defined by adding the relations $u_{\sigma_{\alpha}}^2=0$ for any simple root $\alpha$. 
\end{rem}
\begin{proof}
Let $\sigma_1'\sigma_2'\cdots\sigma_l'$ be a different reduced expression for $\sigma_1\cdots\sigma_l$ that only differs in positions $p+1,\dots,p+m$, with
\[\sigma_{p+1},\cdots,\sigma_{p+m}=\sigma_{\alpha},\sigma_{\beta},\sigma_{\alpha},\sigma_{\beta},\cdots\]
and
\[\sigma_{p+1}',\cdots,\sigma_{p+m}'=\sigma_{\beta},\sigma_{\alpha},\sigma_{\beta},\sigma_{\alpha},\cdots\]
for some simple roots $\alpha,\beta$, and $m=m(\alpha,\beta)$. It is well-known that every reduced expression can be obtained from any other by a series of transformations of this type.

Since $(1+\sigma_i\sigma u)=\sigma_i(1+\sigma u)$, we have
\[R_{\alpha_1,\cdots,\alpha_l}=R_{\alpha_1,\cdots,\alpha_i}\sigma_1\sigma_2\cdots\sigma_iR_{\alpha_i+1,\cdots,\alpha_l}.\]
Hence,
\[R_{\alpha_1,\cdots,\alpha_l}=R_{\alpha_1,\cdots,\alpha_p}\sigma_1\cdots\sigma_pR_{\alpha,\beta,\alpha,\cdots}\sigma_1\cdots\sigma_p\sigma_{\alpha}\sigma_{\beta}\sigma_{\alpha}\cdots R_{\alpha_{p+m+1},\cdots,\alpha_l},\]
so we only have to prove 
\[R_{\alpha,\beta,\alpha,\cdots}=R_{\beta,\alpha,\beta,\cdots}.\]
We show it case by case. We use letter $\alpha_i$ for $\alpha$, and $\alpha_j$ for $\beta$.

\begin{enumerate}
\item $m=2$. Then $\sigma_i\alpha_j=\alpha_j, \sigma_j\alpha_i=\alpha_i$. Therefore,
\begin{align*}
R_{\alpha_i,\alpha_j}
&=\left(1+\frac{\alpha_i}{\hbar}u_{\sigma_i}\right)\left(1+\frac{\sigma_i\alpha_j}{\hbar}u_{\sigma_j}\right)\\
&=1+\frac{\alpha_i}{\hbar}u_{\sigma_i}+\frac{\alpha_j}{\hbar}u_{\sigma_j}+\frac{\alpha_i\alpha_j}{\hbar^2}u_{\sigma_i\sigma_j}\\
&=R_{\alpha_j,\alpha_i}.
\end{align*}

\item $m=3$. Then 
\[\sigma_i\alpha_j=\sigma_j\alpha_i=\alpha_i+\alpha_j.\]
Therefore,
\begin{align*}
R_{\alpha_i,\alpha_j,\alpha_i}
=&\left(1+\frac{\alpha_i}{\hbar}u_{\sigma_i}\right)\left(1+\frac{\sigma_i\alpha_j}{\hbar}u_{\sigma_j}\right)\left(1+\frac{\sigma_i\sigma_j\alpha_i}{\hbar}u_{\sigma_i}\right)\\
=&1+\frac{\alpha_i\alpha_j}{\hbar^2}+\frac{\alpha_i+\alpha_j}{\hbar}(u_{\sigma_i}+u_{\sigma_j})\\
&+\frac{\alpha_i(\alpha_i+\alpha_j)}{\hbar^2}u_{\sigma_i\sigma_j}+\frac{\alpha_j(\alpha_i+\alpha_j)}{\hbar^2}u_{\sigma_j\sigma_i}+\frac{\alpha_i\alpha_j(\alpha_i+\alpha_j)}{\hbar^3}u_{\sigma_i\sigma_j\sigma_j}\\
=&R_{\alpha_j,\alpha_i,\alpha_j}.
\end{align*}

\item $m=4$. Without loss of generality, assume $\alpha_i$ is the short root. Then 
\begin{align*}
&\sigma_i\sigma_j\sigma_i\alpha_j=\alpha_j, & 
&\sigma_j\sigma_i\sigma_j\alpha_i=\alpha_i,\\
&\sigma_i\sigma_j\alpha_i=\sigma_j\alpha_i=\alpha_i+\alpha_j, &
&\sigma_j\sigma_i\alpha_j=\sigma_i\alpha_j=2\alpha_i+\alpha_j.
\end{align*}
Therefore,
\begin{align*}
R_{\alpha_i,\alpha_j,\alpha_i,\alpha_j}
&=\left(1+\frac{\alpha_i}{\hbar}u_{\sigma_i}\right)\left(1+\frac{\sigma_i\alpha_j}{\hbar}u_{\sigma_j}\right)\left(1+\frac{\sigma_i\sigma_j\alpha_i}{\hbar}u_{\sigma_i}\right)\left(1+\frac{\sigma_i\sigma_j\sigma_i\alpha_j}{\hbar}u_{\sigma_j}\right)\\
&=\left(1+\frac{\alpha_i}{\hbar}u_{\sigma_i}\right)\left(1+\frac{2\alpha_i+\alpha_j}{\hbar}u_{\sigma_j}\right)\left(1+\frac{\alpha_i+\alpha_j}{\hbar}u_{\sigma_i}\right)\left(1+\frac{\alpha_j}{\hbar}u_{\sigma_j}\right)
\end{align*}
and
\begin{align*}
R_{\alpha_j,\alpha_i,\alpha_i,\alpha_j}
&=\left(1+\frac{\alpha_j}{\hbar}u_{\sigma_j}\right)\left(1+\frac{\sigma_j\alpha_i}{\hbar}u_{\sigma_i}\right)\left(1+\frac{\sigma_j\sigma_i\alpha_j}{\hbar}u_{\sigma_j}\right)\left(1+\frac{\sigma_j\sigma_i\sigma_j\alpha_i}{\hbar}u_{\sigma_i}\right)\\
&=\left(1+\frac{\alpha_j}{\hbar}u_{\sigma_j}\right)\left(1+\frac{\alpha_i+\alpha_j}{\hbar}u_{\sigma_i}\right)\left(1+\frac{2\alpha_i+\alpha_j}{\hbar}u_{\sigma_j}\right)\left(1+\frac{\alpha_i}{\hbar}u_{\sigma_i}\right).
\end{align*}

Due to Billey's calculation in \cite{Billey1997}, we only have to compare the coefficients of $1=u_{\sigma_i}^2=u_{\sigma_j}^2, u_{\sigma_i}=u_{\sigma_j}^2u_{\sigma_i}=u_{\sigma_i}u_{\sigma_j}^2$, and $u_{\sigma_j}=u_{\sigma_i}^2u_{\sigma_j}=u_{\sigma_j}u_{\sigma_i}^2$. It is easy to see this by a direct calculation.

\item $m=6$. Without loss of generality, assume $\alpha_i$ is the short root. Then 
\begin{align*}
&\sigma_i\sigma_j\sigma_i\sigma_j\sigma_i\alpha_j=\alpha_j, & 
&\sigma_j\sigma_i\sigma_j\sigma_i\sigma_j\alpha_i=\alpha_i,\\
&\sigma_i\sigma_j\sigma_i\sigma_j\alpha_i=\sigma_j\alpha_i=\alpha_i+\alpha_j, & 
&\sigma_j\sigma_i\sigma_j\sigma_i\alpha_j=\sigma_i\alpha_j=3\alpha_i+\alpha_j,\\
&\sigma_i\sigma_j\sigma_i\alpha_j=\sigma_j\sigma_i\alpha_j=3\alpha_i+2\alpha_j, &
&\sigma_j\sigma_i\sigma_j\alpha_i=\sigma_i\sigma_j\alpha_i=2\alpha_i+\alpha_j.
\end{align*}
Therefore,
\begin{align*}
R_{\alpha_i,\alpha_j,\alpha_i,\alpha_j, \alpha_i, \alpha_j}
=&\left(1+\frac{\alpha_i}{\hbar}u_{\sigma_i}\right)\left(1+\frac{\sigma_i\alpha_j}{\hbar}u_{\sigma_j}\right)\left(1+\frac{\sigma_i\sigma_j\alpha_i}{\hbar}u_{\sigma_i}\right)\\
&\left(1+\frac{\sigma_i\sigma_j\sigma_i\alpha_j}{\hbar}u_{\sigma_j}\right)\left(1+\frac{\sigma_i\sigma_j\sigma_i\sigma_j\alpha_i}{\hbar}u_{\sigma_i}\right)\left(1+\frac{\sigma_i\sigma_j\sigma_i\sigma_j\sigma_i\alpha_j}{\hbar}u_{\sigma_j}\right)\\
=&\left(1+\frac{\alpha_i}{\hbar}u_{\sigma_i}\right)\left(1+\frac{3\alpha_i+\alpha_j}{\hbar}u_{\sigma_j}\right)\left(1+\frac{2\alpha_i+\alpha_j}{\hbar}u_{\sigma_i}\right)\\
&\left(1+\frac{3\alpha_i+2\alpha_j}{\hbar}u_{\sigma_j}\right)\left(1+\frac{\alpha_i+\alpha_j}{\hbar}u_{\sigma_i}\right)\left(1+\frac{\alpha_j}{\hbar}u_{\sigma_j}\right)
\end{align*}
and
\begin{align*}
R_{\alpha_j,\alpha_i,\alpha_j,\alpha_i,\alpha_j,\alpha_i}=&\left(1+\frac{\alpha_j}{\hbar}u_{\sigma_j}\right)\left(1+\frac{\alpha_i+\alpha_j}{\hbar}u_{\sigma_i}\right)\left(1+\frac{3\alpha_i+2\alpha_j}{\hbar}u_{\sigma_j}\right)\\
&\left(1+\frac{2\alpha_i+\alpha_j}{\hbar}u_{\sigma_i}\right)\left(1+\frac{3\alpha_i+\alpha_j}{\hbar}u_{\sigma_j}\right)\left(1+\frac{\alpha_i}{\hbar}u_{\sigma_i}\right).
\end{align*}
Similarly to Case (3), we can show the coefficients of the corresponding terms are the same.
\end{enumerate}
\end{proof}

Now we prove the other part.
\begin{prop}
The formula in Theorem \ref{restriction -B} satisfies the properties in Corollary \ref{characterize -}.
\end{prop}
\begin{proof}
\leavevmode
\begin{enumerate}
\item Property (a) follows from Theorem \ref{stable theorem B}.
\item Property (b) follows from the fact 
\[
\{y\alpha|y\alpha\in -R^+,\alpha\in R^+\}=\{-\beta_i|1\leq i\leq l\}.
\]
\item 
Property (c) follows from Proposition \ref{independence for -} as follows. Suppose $y=\sigma_1\cdots\sigma_l$ reduced and $l(y\sigma_{\alpha})=l(y)+1$. Then 
\[R_{\alpha_1,\cdots,\alpha_l,\alpha}=R_{\alpha_1,\cdots,\alpha_l} yR_{\alpha}.\]
Using Equation (\ref{equation P}), we get 
\[\frac{\hbar-y\alpha}{\hbar}\stab_-(w)|_{y\sigma_{\alpha}}=\stab_-(w)|_y+\frac{y\alpha}{\hbar}\stab_-({w\sigma_{\alpha
}})|_y,\]
which is precisely property (3) in corollary \ref{characterize -}.
\end{enumerate}
\end{proof}
This finishes the proof of Theorem \ref{restriction -B}. We give a corollary of it.
\begin{cor}\label{mod h^2 B 1}
\[\stab_-(w)|_y \equiv \left\{\begin{array}{ccc}
\displaystyle (-1)^{l(y)+1}\frac{\hbar\prod\limits_{\alpha\in R^+}\alpha}{y\beta} & \pmod{\hbar^2} & \text{if } w=y\sigma_{\beta}<y \text{ for some } \beta\in R^+,\\
\\
0 & \pmod{\hbar^2} & \text{otherwise}.
\end{array}\right.\]
\end{cor}
\begin{proof}
Theorem \ref{restriction -B} implies that $\stab_-(w)|_y \pmod{\hbar^2}$ is nonzero if and only if $w=\sigma_1\cdots\hat{\sigma_i}\cdots\sigma_l$ for some $i$. Then $w=y\sigma_{\beta}$ with $\beta=\sigma_l\cdots\sigma_{i+1}\alpha_i$ and $\beta_i=-y\beta$. And every element $w=y\sigma_{\beta}$ such that $w<y$ is of the form $\sigma_1\cdots\hat{\sigma_i}\cdots\sigma_l$ for some $i$. Putting these into Equation (\ref{formula restriction -B}) gives the desired result.
\end{proof}

\subsection{Proof of Theorem \ref{restriction B}}
By Corollary \ref{characterize +}, we only have to prove that the formula given by Theorem \ref{restriction B} satisfies those properties.

Let us consider the semidirect product $Q\rtimes W$, where $Q$ is the quotient field of $H_T^*(\text{pt})$. For any element $w\in W$, let $u_w$ denote it. The action of $W$ on $Q$ is induced from the action of $W$ on the Lie algebra of the maximal torus $A$, and $W$ acts trivially on $\hbar$. For example, for any root $\alpha$, $u_w\frac{\hbar+\alpha}{\alpha}=\frac{\hbar+w\alpha}{w\alpha}u_w$. 

For any reduced decomposition $\sigma_{\alpha_1}\cdots\sigma_{\alpha_l}$ of $y$, we define the following function on $W$
\[\xi_{\alpha_1,\alpha_2,\cdots,\alpha_l}(w)=\sum\limits_{\substack{1\leq i_1<i_2<\dots<i_k\leq l\\ 
w=\sigma_{i_1}\sigma_{i_2}\dots\sigma_{i_k}}}\prod\limits_{j=1}^k\frac{\sigma_{i_1}\sigma_{i_2}\dots\sigma_{i_j}\alpha_{i_j}-\hbar}{\sigma_{i_1}\sigma_{i_2}\dots\sigma_{i_j}\alpha_{i_j}}\frac{\hbar^{l-k}}{\prod\limits_{j=0}^k\prod\limits_{i_j<r<i_{j+1}}\sigma_{i_1}\sigma_{i_2}\dots\sigma_{i_j}\alpha_r}.\]
Notice that the element
\[R_{\alpha_1,\alpha_2,\cdots,\alpha_l}:=\prod\limits_{i=1}^l \left(\frac{\hbar}{\alpha_i}+\frac{\hbar+\alpha_i}{\alpha_i}u_{\sigma_i}\right)\]
of $Q\rtimes W$ has the expansion
\begin{equation}\label{equation}
R_{\alpha_1,\alpha_2,\cdots,\alpha_l}=\sum_{w\in W}\xi_{\alpha_1,\alpha_2,\cdots,\alpha_l}(w) u_w.
\end{equation}
Similarly to Proposition \ref{independence for -}, we have
\begin{prop}\label{independence +}
$R_{\alpha_1,\alpha_2,\cdots,\alpha_l}$ does not depend on the choice of the reduced expression for $y$.
\end{prop}
This can be checked case by case as in Proposition \ref{independence for -}. We leave the details to the interested reader. Finally, we can prove the Theorem \ref{restriction B}.
\begin{prop}
The formula in Theorem \ref{restriction B} satisfies the properties in Corollary \ref{characterize +}.
\end{prop}
\begin{proof}
\leavevmode
\begin{enumerate}
\item
Property (a) is clear due to Proposition 8.5.5 in \cite{Springer2010}.
\item 
Let $R(y):=\{\alpha\in R^+|y\alpha<0\}$. Then Lemma 8.3.2 in 
\cite{Springer2010} gives
\[R(y)=\{\alpha_l,\sigma_l\alpha_{l-1},\cdots,\sigma_l\sigma_{l-1}\cdots\sigma_2\alpha_1\}.\]
Therefore
\[\{y\alpha|\alpha\in R^+, y\alpha<0\}=\{\sigma_1\cdots\sigma_j\alpha_j|1\leq j\leq l\}.\]
This implies property (b).
\item
Since
\[R_{\alpha_1,\alpha_2,\cdots,\alpha_l,\alpha}=R_{\alpha_1,\alpha_2,\cdots,\alpha_l}\left(\frac{\hbar}{\alpha}+\frac{\hbar+\alpha}{\alpha}u_{\sigma_{\alpha}}\right),\]
Equation (\ref{equation}) gives
\begin{align*}
\sum_w \xi_{\alpha_1,\alpha_2,\cdots,\alpha_l,\alpha}(w) u_w
&=\sum_{w'} \xi_{\alpha_1,\alpha_2,\cdots,\alpha_l}(w') u_{w'}\left(\frac{\hbar}{\alpha}+\frac{\hbar+\alpha}{\alpha}u_{\sigma_{\alpha}}\right)\\
&=\sum_w \left(\frac{\hbar}{w\alpha}\xi_{\alpha_1,\alpha_2,\cdots,\alpha_l}(w)+\frac{w\alpha-\hbar}{w\alpha}\xi_{\alpha_1,\alpha_2,\cdots,\alpha_l}(w\sigma_{\alpha})\right)u_w.
\end{align*}
Hence
\[\xi_{\alpha_1,\alpha_2,\cdots,\alpha_l,\alpha}(w)=\frac{\hbar}{w\alpha}\xi_{\alpha_1,\alpha_2,\cdots,\alpha_l}+\frac{w\alpha-\hbar}{w\alpha}\xi_{\alpha_1,\alpha_2,\cdots,\alpha_l}(w\sigma_{\alpha})\]
which is precisely property (c).
\end{enumerate}
\end{proof}

As in Corollary \ref{mod h^2 B 1}, modulo $\hbar^2$, we get
\begin{cor}\label{mod h^2 B}
\[\stab_+(y)|_w \equiv \left\{\begin{array}{ccc}
\displaystyle (-1)^{l(y)+1}\frac{\hbar\prod\limits_{\alpha\in R^+}\alpha}{y\beta} & \pmod{\hbar^2} &\text{if } w=y\sigma_{\beta}<y \text{ for some } \beta\in R^+,\\
\\
0 & \pmod{\hbar^2} & \text{otherwise}.
\end{array}\right.\]
\end{cor}
The follows from the proof of Corollary \ref{mod h^2 B 1} and Theorem \ref{restriction B}. 

\section{Applications}

\subsection{Restriction of stables basis for $T^*\mathcal{P}$}

In this subsection, we extend Theorems \ref{restriction -B} and \ref{restriction B} to $T^*\mathcal{P}$ case. In type $A$, these formulas were also obtained in \cite{Shenfeld2013} via a process called Abelianization. 

Before we state the theorem, we record a useful lemma from \cite{Bernstein1973a}.
\begin{lem}
Each coset $W/W_P$ contains exactly one element of minimal length, which is characterized by the property that it maps $I$ into $R^+$.
\end{lem}

Let $\pi$ be the projection map from $\mathcal{B}$ to $\mathcal{P}$, and $\Gamma_{\pi}$ be its graph. Then the conormal bundle to $\Gamma_{\pi}$ in $\mathcal{B}\times \mathcal{P}$ is a Lagrangian submanifold of $T^*(\mathcal{B}\times \mathcal{P})$. 
\[\xymatrix{
T_{\Gamma_{\pi}}^*(\mathcal{B}\times \mathcal{P}) \ar[r]^-{p_1} \ar[d]_{p_2} & T^*\mathcal{B}  \\
T^*\mathcal{P} \\}. \]

Let $D_1=p_{2*}p_1^*$, $D_2=p_{1*}p_2^*$ be the maps induced by the correspondence $T_{\Gamma_{\pi}}^*(\mathcal{B}\times \mathcal{P})$. Recall we have an embedding of $H_T^*(T^*\mathcal{B})$ into $F(W,Q)$ by restricting every cohomology class to fixed points. Since the fixed point set $(T^*\mathcal{P})^T$ is in one-to-one correspondence with $W/W_P$, we can embed
$H_T^*(T^*\mathcal{P})$ into $F(W/W_P,Q)$. For any $y\in W$, let $\bar{y}$ denote the coset $yW_P$. Recall that for any cohomology class $\alpha\in H_T^*(T^*\mathcal{P})$, let $\alpha|_{\bar{y}}$ denote the restriction of $\alpha$ to the fixed point $yP$.

Define a map 
\[A_1:F(W,Q)\rightarrow F(W/W_P,Q)\]
as follows: for any $\psi\in F(W,Q)$,
\[A_1(\psi)(\bar{w})=\sum_{\bar{z}=\bar{w}}\frac{\psi(z)}{\prod\limits_{\alpha\in R^+_P}(-z\alpha)}.\] 
Then as Proposition \ref{commutative diagram 1}, we have
\begin{prop}\label{diagram 2}
The diagram
\[\xymatrix{
H_T^*(T^*\mathcal{B}) \ar@{^{(}->}[r]
 \ar[d]_{D_1} & F(W,Q) \ar[d]^{A_1} \\
H_T^*(T^*\mathcal{P}) \ar@{^{(}->}[r] & F(W/W_P,Q)\\} \]
commutes.
\end{prop}
\begin{proof}
The proof is almost the same as that of Proposition \ref{commutative diagram 1}. We show the two paths agree on the fixed point basis.

By the definition of $A_1$,
\begin{align*}
A_1(\iota_{y*}1)(\bar{w})&=\delta_{\bar{y},\bar{w}}\frac{e(T_yT^*\mathcal{B})}{\prod\limits_{\alpha\in R^+_P}(-y\alpha)}\\
&=\delta_{\bar{y},\bar{w}}e(T_{\bar{y}}T^*\mathcal{P})\prod\limits_{\alpha\in R^+_P}(y\alpha-\hbar).
\end{align*}
By localization, 
\begin{align*}
D_1(\iota_{y*}1)&=\sum\limits_{\bar{w}}\frac{(D_1(\iota_{y*}1), \iota_{\bar{w}*}1)}{e(T_{\bar{w}}T^*\mathcal{P})}\iota_{\bar{w}*}1\\
&=\prod\limits_{\alpha\in R^+_P}(y\alpha-\hbar)\iota_{\bar{y}*}1,
\end{align*}
where $\iota_{\bar{y}}$ is the inclusion of the fixed point $yP$ into $T^*\mathcal{P}$. Hence 
\begin{align*}
D_1(\iota_{y*}1)|_{\bar{w}}&=\delta_{\bar{y},\bar{w}}e(T_{\bar{y}}T^*\mathcal{P})\prod\limits_{\alpha\in R^+_P}(y\alpha-\hbar)\\
&=A_1(\iota_{y*}1)(\bar{w})
\end{align*}
as desired.
\end{proof}

If we apply this proposition to the stable basis, we get the following important corollary. 
\begin{cor}\label{restriction P 1}
The restriction formula of the stable basis of $T^*\mathcal{P}$ is given by
\[\stab_\pm(\bar{y})|_{\bar{w}}=\sum_{\bar{z}=\bar{w}}\frac{\stab_\pm(y)|_z}{\prod\limits_{\alpha\in R^+_P}z\alpha}.\]
\end{cor}
\begin{proof}
As in Lemma \ref{image of stab}, 
\[(D_1(\stab_{\pm}(y)),\stab_{\mp}(\bar{w}))\]
is a constant, so we can let $\hbar=0$. Then $\stab_+(\bar{y})|_{\bar{w}}$ is nonzero if and only if $\bar{y}=\bar{w}$. A simple localization gives
\[D_1(\stab_{\pm}(y))=(-1)^k\stab_{\pm}(\bar{y}),\]
where $k=\dim \mathcal{B} - \dim \mathcal{P}=|R^+\setminus R^+_P|$. Applying Proposition \ref{diagram 2} to $\stab_{\pm}(y)$ yields the result.
\end{proof}

As in the $T^*\mathcal{B}$ case, modulo $\hbar^2$ we get 

\begin{cor}\label{mod h^2 P}
Let $y$ be a minimal representative of the coset $yW_P$. Then
\[\stab_+(\bar{y})|_{\bar{w}} \equiv \left\{\begin{array}{ccc}
\displaystyle (-1)^{l(y)+1}\frac{\hbar\prod\limits_{\alpha\in R^+}\alpha}{y\beta\prod\limits_{\alpha\in R^+_P} y\sigma_{\beta}\alpha} & \pmod{\hbar^2} & \text{if } \bar{w}=\overline{y\sigma_{\beta}} \text{ and } y\sigma_{\beta}<y \text{ for some } \beta\in R^+,\\
\\
0 & \pmod{\hbar^2} & \text{otherwise}.
\end{array}\right.\]
\end{cor}
\begin{proof}
Assume $y=\sigma_1\sigma_2\cdots\sigma_l$ is a reduced decomposition. Because of Corollary \ref{mod h^2 B} and Corollary \ref{restriction P 1}, we only have to show: if $i< j$, then
\[\overline{\sigma_1\cdots\hat{\sigma_i}\cdots\sigma_l}\neq \overline{\sigma_1\cdots\hat{\sigma_j}\cdots\sigma_l}.\]
Assume the contrary. Then there exists an element $w\in W_P$ such that 
\[\sigma_1\cdots\hat{\sigma_i}\cdots\sigma_l=\sigma_1\cdots\hat{\sigma_j}\cdots\sigma_lw.\]
Then
 \[y=\sigma_1\cdots\sigma_i\cdots\sigma_l=\sigma_1\cdots\hat{\sigma_i}\cdots\hat{\sigma_j}\cdots\sigma_lw,\]
which is contradictory to the fact that $y$ is minimal.
\end{proof}

Using the map $D_2$, we can get another restriction formula for the stable basis of $T^*\mathcal{P}$. Define a map
\[A_2:F(W/W_P,Q)\rightarrow F(W,Q)\]
as follows: for any $\psi\in F(W/W_P,Q)$,
\[A_2(\psi)(z)=\psi(\bar{z})\prod\limits_{\alpha\in R^+_P}(z\alpha-\hbar).\]
Then we have the following commutative diagram.

\begin{prop}\label{diagram 3}
The diagram
\[\xymatrix{
H_T^*(T^*\mathcal{P}) \ar@{^{(}->}[r]
 \ar[d]_{D_2} & F(W/W_P,Q) \ar[d]^{A_2} \\
H_T^*(T^*\mathcal{B}) \ar@{^{(}->}[r] & F(W,Q)\\} \]
commutes.
\end{prop}
\begin{proof}
The proof is almost the same as in Proposition \ref{commutative diagram 1}. We show the two paths agree on the fixed point basis.

By the definition of $A_2$,
\[A_2(\iota_{\bar{y}*}1)(w)=\delta_{\bar{y},\bar{w}}e(T_{\bar{y}}T^*\mathcal{P})\prod\limits_{\alpha\in R^+_P}(w\alpha-\hbar).\]
By localization, 
\begin{align*}
D_2(\iota_{\bar{y}*}1)&=\sum\limits_{w}\frac{(D_2(\iota_{\bar{y}*}1), \iota_{w*}1)}{e(T_wT^*\mathcal{B})}\iota_{w*}1\\
&=\sum\limits_{\bar{w}=\bar{y}}\frac{1}{\prod\limits_{\alpha\in R^+_P}(-w\alpha)}\iota_{w*}1.
\end{align*}
Hence 
\begin{align*}
D_2(\iota_{\bar{y}*}1)|_w&=\delta_{\bar{y},\bar{w}}\frac{e(T_wT^*\mathcal{B})}{\prod\limits_{\alpha\in R^+_P}(-w\alpha)}\\
&=A_2(\iota_{\bar{y}*}1)(w)
\end{align*}
as desired.
\end{proof}

If we apply this diagram to the stable basis, we get
\begin{cor}\label{restriction P 2}
We have the restriction formula for the stable basis of $T^*\mathcal{P}$
\[\stab_{\pm}(\bar{y})|_{\bar{z}}=\frac{\sum_{\bar{w}=\bar{y}}\stab_{\pm}(w)|_z}{\prod\limits_{\alpha\in R^+_P}(z\alpha-\hbar)}.\]
\end{cor}
\begin{proof}
As in Lemma \ref{image of stab}, 
\[(D_2(\stab_{\pm}(\bar{y})),\stab_{\mp}(w))\]
is a constant, so we can let $\hbar=0$. A simple localization calculation gives
\[D_2(\stab_{\pm}(\bar{y}))=\sum_{\bar{w}=\bar{y}}\stab_{\pm}(w).\]
Applying Proposition \ref{diagram 3} to $\stab_{\pm}(y)$ yields the result.
\end{proof}

Modulo $\hbar^2$, we get

\begin{cor}\label{mod h^2 P 2}
Let $y$ be a minimal representative of the coset $yW_P$. Then
\[\stab_-(\bar{w})|_{\bar{y}}\equiv\left\{\begin{array}{ccc}
\displaystyle (-1)^{l(y)+1}\frac{\hbar\prod\limits_{\alpha\in R^+}\alpha}{y\beta\prod\limits_{\alpha\in R^+_P} y\alpha} &\pmod{\hbar^2} &\text{if } \bar{w}=\overline{y\sigma_{\beta}} \text{ and } y\sigma_{\beta}<y \text{ for some } \beta\in R^+,\\
\\
0 & \pmod{\hbar^2} & \text{otherwise}.\\
\end{array}\right.\]
\end{cor}
\begin{proof}
This follows directly from Corollary \ref{mod h^2 B 1} and Corollary \ref{restriction P 2} and the proof of
Corollary \ref{mod h^2 P}. 
\end{proof}

\subsection{Restriction of Schubert varieties}

In \cite{Billey1997}, Billey gave a restriction formula for Schubert varieties in $G/B$, and Tymoczko  generalized it to $G/P$ in \cite{Tymoczko2009}. In this section, we will show Billey's formula from Theorem \ref{restriction -B} by a limiting process and generalize it to $G/P$ case in two ways.

Let us recall Billey's formula. Let $B^-$ be the opposite Borel subgroup to $B$. Then $\overline{B^-wB/B}$ is the Schubert variety in $G/B$ of dimension $\dim G/B-l(w)$, and as $w \in W$ varies they form a basis of $H_A^*(G/B)$. The formula is
\begin{thm}[\cite{Billey1997}]\label{Billey}
Let $y=\sigma_1\sigma_2\cdots\sigma_l$ be a reduced decomposition. Then we have
\[[\overline{B^-wB/B}]|_y=\sum\limits_{\substack{1\leq i_1<i_2<\dots<i_k\leq l\\ 
w=\sigma_{i_1}\sigma_{i_2}\dots\sigma_{i_k} \text{reduced}}}\beta_{i_1}\cdots\beta_{i_k}\]
where $\beta_i=\sigma_1\cdots\sigma_{i-1}\alpha_i$. 
\end{thm}
\begin{proof}
By the construction of the stable basis, we have 
\[[\overline{B^-wB/B}]|_y=\pm\lim_{\hbar\rightarrow \infty}\frac{\stab_-(w)|_y}{(-\hbar)^{\dim\overline{B^-wB/B}}}.\]
The sign only depends on $w$, and can be determined by substituting $y=w$ as follows: the left hand side is
\[[\overline{B^-wB/B}]|_w=\prod\limits_{\alpha\in R^+\cap wR^-}\alpha,\]
whereas the limit on the right is
\begin{align*}
\lim_{\hbar\rightarrow \infty}\frac{\stab_-(w)|_w}{(-\hbar)^{n-l(w)}}&=\lim_{\hbar\rightarrow \infty}\frac{\prod\limits_{\alpha\in R^+, w\alpha>0}(w\alpha-\hbar)\prod\limits_{\alpha\in R^+, w\alpha<0}w\alpha}{(-\hbar)^{n-l(w)}}\\
&=\prod\limits_{\alpha\in R^+,w\alpha<0}w\alpha\\
&=(-1)^{l(w)}[\overline{B^-wB/B}]|_w.
\end{align*}
Hence the sign is $(-1)^{l(w)}$. Now the formula follows from Theorem \ref{restriction -B}.
\end{proof}
\begin{rem}
The proof of Theorems \ref{restriction B} and \ref{restriction -B} is inspired by Billey's proof of Theorem \ref{Billey}. Using Theorem \ref{restriction B} we can also get a restriction formula for $[\overline{ByB/B}]|_w$. 
\end{rem}

A similar limiting process for $T^*\mathcal{P}$ yields the restriction formula for Schubert varieties in $G/P$. Recall that if $w$ is minimal, then $\overline{B^-\bar{w}P/P}$ is the Schubert variety in $G/P$ of dimension $\dim G/P-l(w)$, and as $w$ runs through the minimal elements they form a basis of $H_A^*(G/P)$.

\begin{thm}[\cite{Tymoczko2009}]\label{Schubert P}
Let $y,w$ be minimal representatives of $yW_P, wW_P$ respectively. Then we have
\[[\overline{B^-\bar{w}P/P}]|_{\bar{y}}=[\overline{B^-wB/B}]|_y.\] 
\end{thm}

\begin{rem}
Tymoczko's generalization in \cite{Tymoczko2009} does not require $y$ to be minimal. We will give two proofs of it. The first proof only works for minimal $y$, while the second works for any $y$.
\end{rem}

\begin{proof}
Similarly to the proof of Theorem \ref{Billey}, we have 
\[[\overline{B^-\bar{w}P/P}]|_{\bar{y}}=(-1)^{l(w)}\lim_{\hbar\rightarrow \infty}\frac{\stab_-(\bar{w})|_{\bar{y}}}{(-\hbar)^{m-l(w)}},\]
where $m=\dim G/P$. Then the formula follows from Corollary \ref{restriction P 2}.
\end{proof}
 
We give another much simpler proof using a commutative diagram.
Define a map
\[A_3:F(W/W_P,Q)\rightarrow F(W,Q)\]
as follows:
for any $\psi\in F(W/W_P,Q),$
\[A_3(\psi)(z)=\psi(\bar{z}).\]
Then we have the following commutative diagram.

\begin{prop}\label{diagram 4}
The diagram
\[\xymatrix{
H_A^*(G/P) \ar@{^{(}->}[r] \ar[d]_{\pi^*} & F(W/W_P,Q) \ar[d]^{A_3} \\
H_A^*(G/B) \ar@{^{(}->}[r] & F(W,Q)\\} \]
commutes, where $\pi$ is the projection from $G/B$ onto $G/P$.
\end{prop}
\begin{proof}
We check on the fixed point basis.
\[\pi^*\iota_{\bar{y}*}1|_w=\iota_w^*\pi^*\iota_{\bar{y}*}1
=(\pi\circ \iota_w)^*\iota_{\bar{y}*}1=\iota_{\bar{w}}^*\iota_{\bar{y}*}1
=\delta_{\bar{y},\bar{w}}e(T_{\bar{y}}G/P).\]

By definition of $A_3$,
\[A_3(\iota_{\bar{y}*}1)(w)=\iota_{\bar{y}*}1|_{\bar{w}}=\delta_{\bar{y},\bar{w}}e(T_{\bar{y}}G/P),\]
as desired.
\end{proof}

Since $\pi^*([\overline{B^-\bar{w}P/P}])=[\overline{B^-wB/B}]$ if $w$ is minimal (see \cite{Fulton2007}), Proposition \ref{diagram 4} gives
\[
A_3([\overline{B^-\bar{w}P/P}])(y)=[\overline{B^-\bar{w}P/P}]|_{\bar{y}}=\pi^*([\overline{B^-\bar{w}P/P}])|_y=[\overline{B^-wB/B}]|_y,
\]
which is just Theorem \ref{Schubert P} without any restrictions on $y$.

\bibliographystyle{plain}
\bibliography{stablebasis}

\end{document}